\documentclass[12pt,leqno]{amsart}
\usepackage{amscd}
\usepackage{color}
\usepackage{amssymb}
\usepackage{amsfonts}
\usepackage{latexsym}
\usepackage{verbatim}

\theoremstyle{plain}
\newtheorem{theorem}{Theorem}[section]

\newtheorem{lemma}[theorem]{Lemma}
\newtheorem{prop}[theorem]{Proposition}
\newtheorem{cor}[theorem]{Corollary}
\newtheorem{rem}[theorem]{Remark}
\newtheorem{ex}[theorem]{Example}

\renewcommand{\b}{\begin{equation}}
\newcommand{\e}{\end{equation}}

\newcommand\C{{\mathbb C}}

\newcommand\Z{{\mathbb Z}}

\title[The second Ricci flow on complex parallelizable manifolds]{The behavior of the second Ricci flow on complex parallelizable manifolds}

\makeatletter
\@namedef{subjclassname@2020}{%
 \textup{2020} Mathematics Subject Classification}
\makeatother
\thanks{This work was supported by GNSAGA of INdAM and by the project PRIN \lq\lq Differential-geometric aspects of manifolds via Global Analysis''}

\address{(Lucio Bedulli) Dipartimento di Ingegneria e Scienze dell'Informazione e Matematica, 
Universit\`a dell'Aquila, via Vetoio, 67100 L'Aquila, Italy}
\email{lucio.bedulli@univaq.it}

\address{(Luigi Vezzoni) Dipartimento di Matematica G. Peano \\ Universit\`a di Torino\\
Via Carlo Alberto 10\\
10123 Torino\\ Italy}
\email{luigi.vezzoni@unito.it}

\begin{document}
\author{Lucio Bedulli $\&$ Luigi Vezzoni}
\maketitle

\date{\today}
\begin{abstract}
We study the flow of Hermitian metrics governed by the second Chern-Ricci form on a compact complex manifold. The flow belongs to the family of Hermitian curvature flows introduced by Streets and Tian and it was considered by Lee in order to study compact Hermitian manifolds with almost negative Chern bisectional curvature. We show a regularity result on compact complex parallelizable manifolds and we prove that Chern-flat metrics are dynamically stable.    
\end{abstract}
\section{Introduction}
In 	\cite{HCF} Streets and Tian introduced a family of parabolic flows on complex manifolds which evolve an initial  Hermitian metric as 
\begin{equation}\label{HCF}
%\partial_tg_{r\bar s}=-\tilde R_{r	\bar s}+Q_{r\bar s}\,,
\partial_tg=-\tilde R+Q\,,
\end{equation}
where $\tilde R$ is the second Chern-Ricci-form and $Q$ is a quadratic form in the torsion of $g$. A key observation in \cite{HCF} is that equation \eqref{HCF} is strictly parabolic for any choice of $Q$ and the standard theory ensures the existence of a short-time solution.  For a suitable choice of $Q$ flow \eqref{HCF} preserves the pluriclosed condition $\partial \bar \partial \omega=0$; in this case the corresponding equation is called {\em pluriclosed flow} and it has been widely investigated in the literature (see e.g. \cite{mario,jordan,ST,ST3,ST4} and the references therein). In \cite{YU1} Ustinovskiy proved that for a different choice of $Q$ the corresponding Hermitian curvature flow, called the {\em positive Hermitian curvature flow}, preserves the Griffiths positivity of the metric and used this to 
obtain a characterization of the projective space in terms of  Chern curvature.  In \cite{AF9} Fei and Phong proved that  on conformally balanced manifolds the positive Hermitian curvature  flow is conformally equivalent to the type IIB flow introduced by Phong, Picard and Zhang in  \cite{AF8,AF7} and studied in \cite{AF2,estimates,AF3,AF4,AF5,AF6}. This gives a strong geometric motivation for studying the positive Hermitian curvature flow. 
In \cite{Lee, Lee2} Lee considered the Hermitian curvature flow 
\begin{equation}\label{Q=0}
\partial_tg=-\tilde R
\end{equation}
obtained by setting $Q=0$ and in \cite{Lee} he used the flow to show that a compact Hermitian manifold with Griffiths non-positive Chern curvature and negative first Chern Ricci curvature at one point admits a K\"ahler-Einstein metric.

 An equation formally similar to \eqref{Q=0} 
emerged also in the study of the pluriclosed flow. 
Any pluriclosed metric $g$ induces a canonical generalized Hermitian metric $G$ and the pluriclosed flow once regarded as a flow of generalized Hermitian metrics satisfies 
$$
G^{-1}\partial_tG=-\tilde R^{G}_g\,,
$$ 
see \cite{jordan,mario}.

\medskip 
In the present paper we study equation \eqref{Q=0} on compact complex parallizable manifolds, i.e., on compact quotients of complex Lie groups by lattices. In anology with the main theorem in  \cite{mario} which states the long-time existence and convergence of the pluriclosed flow on a compact complex manifold with a Bismut-flat Hermitian metric, it is natural to expect  that \eqref{Q=0} on compact parallelizable manifolds has a long-time solution for every initial datum which converges to a Chern-flat metric. 
We prove the long-time existence under an extra assumption. For a Hermitian metric $g$ on a complex parallelizable manifold we define the tensor $\Gamma$ whose components $\Gamma_{ij}^k$ with respect to a left-invariant $(1,0)$-frame are the Christoffel symbols of the Chern connection of $g$ and let 
$$
w_{k}:=\Gamma_{kr}^{r}\,. 
$$
The tensors $\Gamma$ and $w$ are well-defined since their definitions do not depend on the choice of the frame.   

\medskip
Our first result is the following 
 
\begin{theorem}[Long-time existence]\label{main1}
Let $g$ be a maximal time solution to the Hermitian curvature flow \eqref{Q=0} on a compact complex parallelizable manifold. Assume that $|w|_{g}$ is uniformly bounded, then  $g$ is defined in $M\times [0,\infty)$. 
\end{theorem} 

The theorem and its proof are inspired by \cite{estimates} where it is  proved the long-time existence of the type IIB flow if the trace of torsion is bounded. 
%For every Hermitian curvature flow, once you have a C^1 estimate on g and a non-degeneracy estimate of the metric like C^{-1} g_0 \leq g \leq C g_0, then by this general PDE theorem of Ladyzenskaya et al (followed by application of the Schauder estimates) you get all higher order estimates. 
%uniform a priori $C^1$-bound

The key point is that since Hermitian curvature flows are quasi-linear equations, a $C^1$-estimate for the solutions together with a non-degeneracy estimate of the metric is enough to have the long time existence (see \cite{S2,{AF7}}) and we prove that  in our geometric setting the non-degneracy estimate always holds and the bound on $|w|_g$ is sufficient to have the $C^1$-estimate. 

The proof is obtained in sections \ref{2} and \ref{3} 
and the computations roughly follow the
ideas of Yau's $C^3$-estimate in the proof of the Calabi conjecture \cite{Yau}. 
 
\medskip 
A similar setting is considered in \cite{S2,mario} where  the pluriclosed flow is studied on tori and Bismut-flat manifolds. In both cases it is proved that under suitable assumptions on the initial datum, the flow has a long time solution which converges to a Bismut-flat metric. Moreover, on compact complex parallelizable manifolds the first Bott-Chern class is vanishing, since they always admit Chern-flat metrics; hence \cite[Theorem 1]{Gill} implies that the Chern-Ricci flow \cite{TW} on a compact complex parallelizable manifold has always a long-time solution which converges to a Chern-flat metric for every initial datum.

About the convergence of flow \eqref{Q=0} to a Chern-flat metric in our setting, we prove in section \ref{conv}  that this happens when the initial metric is diagonal with respect to a left-invariant frame. Moreover, we establish the following stability result 

\begin{theorem}[Stability]\label{stability}
Let $M$ be a compact complex parallelizable manifold of complex dimension $n$ and let $\hat g$ be a Chern-flat Hermitian metric on $M$. Then there exists $\delta_0 >0$ such that for every Hermitian metric $g_0$  on $M$ satisfying
$$
\|g_0-\hat g\|_{H^{n+3}_{\hat g}}\leq \delta_0\,,
$$
then the maximal time solution $g$ to \eqref{Q=0} with starting point $g_0$ is defined in $M\times [0,\infty)$ and converges in $C^{\infty}$-topology to a Chern flat metric as $t \to \infty$.  
\end{theorem} 
In \cite{HCF} it is proved that every Hermitian curvature flow is dynamically stable around K\"ahler-Einstein metrics with nonpositive scalar curvature. This theorem cannot be applied directly in order to prove Theorem \ref{stability} since $\hat g$ is not in general K\"ahler. However, since 
%we are considering the Hermitian curvature flow with $Q=0$ and 
 $\hat g$ is Chern-flat, we have that the first variation $L$ of the operator $g\mapsto -\tilde R(g)$ at $\hat g$ is the Chern Laplacian of $\hat g$
%$$
%L=\Delta_{\hat g}
%$$ 
which, since $\hat g$ is balanced (i.e. its fundamental form is coclosed), it is seminegative definite and self-adjoint with respect to the $L^2$-product induced by $\hat g$.  Therefore Chern-flat metrics are linearly stable. The proof of Theorem \ref{stability} follows the approach of similar results in literature (see e.g. 
\cite{StabilityPhong,BVadv,stabilitytypeIIA,reto,sesum,HCF} and the references therein). 

\bigskip
\noindent {\bf Acknowledgement.} The authors are grateful to Jeffrey Streets and James Stanfield  for useful conversations and observations. Moreover the authors would like to thank an anonymous referee for pointing out some inaccuracies and helping to improve the presentation of the paper. 

\section{Preliminaries}\label{2}
Let $M$ be a complex manifold equipped with a Hermitian metric $g$.  If $g$ is not K\"ahler, i.e. if the fundamental form $\omega$ of $g$ is not closed, then the Levi-Civita connection of $g$ does not preserve the complex structure and its role is often replaced by a Hermitian connection, i.e. a connection which preserves both the complex and the Riemannian metric. On a Hermitian manifold there are always many Hermitian connections and canonical choices can be obtained by imposing conditions on the torsion. The Chern connection is then defined as the unique Hermitian  connection $\nabla$ on $(M,g)$ whose torsion $T$ has vanishing $(1,1)$-component. The Christoffel symbols of $\nabla$ read as 
$$
\Gamma_{ij}^k=g^{\bar rk}g_{j	\bar r,i}
$$
and the components of the torsion tensor of $\nabla$ are 
$$
T_{ij}^k=g^{\bar rk}g_{j	\bar r,i}-g^{\bar rk}g_{i	\bar r,j}\,.
$$
Moreover, denoting by $R$ the curvature tensor of $\nabla$, we have 
\begin{equation}\label{tildeR}
R_{k\bar l r	\bar s  }=-g_{u\bar s}\nabla_{\bar l}\Gamma_{kr}^{u}=-g_{r \bar s,k\bar l}+g^{\bar b a}g_{a\bar s,\bar l }g_{r	\bar b,k}\,. 
\end{equation}
We can consider the two Ricci-type tensors 
$$
{\rm Ric}_{r\bar s}=g^{\bar lk}R_{r\bar sl\bar k }\,,\quad 
\tilde R_{r\bar s}=g^{\bar lk}R_{l\bar k r\bar s}\,.
$$
called the {\em first} and {\em second Chern-Ricci tensor}, respectively.
In the non-K\"ahler case these two tensors do not in general agree.  
In particular for the second Chern-Ricci tensor we have
\begin{equation}\label{RicDelta}
\tilde R_{r\bar s}=-\Delta_g g_{r \bar s}+g^{\bar lk}g^{\bar b a}g_{a\bar s,\bar l }g_{r	\bar b,k}\,,
\end{equation}
where $\Delta_g$ is the Chern Laplacian of the metric $g$.

\medskip 
Next we focus on the case when $M$ is a complex {\em parallelizable manifold},  i.e. when $M$  admits a global holomorphic coframe.
%$\{\zeta^1,\dots,\zeta^n\}$. 
It is a result of Wang \cite{wang} that a {\it compact} complex  manifold $M$ is complex parallelizable if and only if it is the quotient  of a complex Lie group $G$ by a lattice $\Gamma$. Indeed, if $M$ is the quotient of a simply connected complex Lie group $G$ by a lattice, any left-invariant $(1,0)$-coframe on $G$ is holomorphic and induces a global holomorphic coframe on $M$; on the the other hand if $M=G/\Gamma$ is complex parallelizable any global holomorphic coframe on $M$ gives a left-invariant $(1,0)$-coframe on $G$. 

Let $M$ be a complex parallelizable manifold with a fixed global holomorphic coframe $\{\zeta^1,\dots,\zeta^n\}$. Then its dual $(1,0)$-frame 
 $\{Z_1,\dots,Z_n\}$ satisfies     
\begin{equation}\label{zk}
[Z_k,Z_{\bar l}]=0\quad \mbox{ for every } k,l\,.  
\end{equation}
Moreover a Hermitian metric $g$ on $M$ lifts to a left-invariant metric on the universal cover of $M$ if and only if
its components  $g_{i\bar j}$ with respect to the frame $\{Z_1,\dots,Z_n\}$ are constant.  
In such a case the holonomy of the Chern connection of $g$ is vanishing. That can be observed as follows:  
$$
g(\nabla_{k}Z_i,Z_{\bar l})=Z_kg(Z_i,Z_{\bar l})-g(Z_i,\nabla_kZ_{\bar l})=
-g(Z_i,\nabla_{\bar l}Z_{k})=0
$$
for every $i,k,l=1,\dots,n$, which implies 
$$
\nabla_{k}Z_i=0\,,\quad \mbox{ for all } k,i=1,\dots,n\,.
$$ 
In particular $g$ is Chern-flat and the torsion of the Chern connection of $g$ satisfies  
$$
T(Z_i,Z_j)=-[Z_i,Z_j]\,. 
$$
Moreover $g$ is {\em balanced}, i.e. the torsion $T$ satisfies
$$
T_{ir}^r=0\,.
$$
This is equivalent to require that the fundamental form of $g$ is co-closed. 
  
\begin{ex}{\em 
The basic example of complex parallelizable manifold is the Iwasawa manifold 
$$
M={\rm Nil}_3(\C)/\Gamma
$$
which is defined as the quotient of the complex $3$-dimensional Heisenberg group
$$
{\rm Nil}_3(\C):=\left\{\left(\begin{array}{ccc}
1 &z^1  & z^2\\
0 & 1 &z^3\\
0   &0     &1
\end{array}
\right)\,\,:\,\, z_1,z_2,z_3\in \mathbb{C} \right\}
$$
by the subgroup $\Gamma$ of elements of ${\rm Nil}_3(\C)$  whose entries are Gauss integers.  The Iwasawa manifold has a canonical complex structure induced by $\C^3$ and a global $(1,0)$-coframe $\{\zeta^1,\zeta^2,\zeta^3\}$ satisfying
$$
d\zeta^1=d\zeta^2=0\,,\quad d\zeta^3=\zeta^{1}\wedge \zeta^2\,.
$$
Since ${\rm Nil}_3(\C)$ is nilpotent, the Iwasawa manifold is a nilmanifold and it has a natural structure of principal torus bundle over a torus of complex dimension $2$.  
}
\end{ex}

\begin{ex}
{\em $3$-dimensional complex unimodular Lie groups are classified. Beside the abelian group, there are only three simply-conneted complex unimodular Lie groups in dimension 3: 
${\rm SL}(2,\mathbb{C})$, ${\rm Nil}_3(\C)$, $S_{3,-1}$. 

\medskip 
\noindent ${\rm SL}(2,\mathbb{C})$  is a simple Lie group and has left-invariant $(1,0)$-frame ${\\Z_1,Z_2,Z_3}$ satisfying 
$$
[Z_1,Z_2]=Z_3\,,\quad [Z_1,Z_3]=-Z_2\,,\quad [Z_2,Z_3]=Z_1\,,
$$
and the other brackets are zero. 
In matrix notation we have 
$$
Z_1=\frac12 \left(\begin{array}{cc}
0 &i\\
i & 0
\end{array}
\right)\,,\quad  Z_2=\frac12 \left(\begin{array}{cc}
0 &1\\
-1 & 0
\end{array}
\right)\,,\quad 
Z_3=\frac12 \left(\begin{array}{cc}
-i &0\\
0 & i
\end{array}
\right)\,.
$$

\medskip 
\noindent
$S_{3,-1}$ is $2$-step solvable (the notation comes form the classification of $6$-dimensional real solvable Lie group with a left-invariant complex structure). The group has a left-invariant $(1,0)$-frame satisfying
$$
[Z_1,Z_2]=Z_2\,,\quad  [Z_1,Z_3]=-Z_3
$$
and the other brackets are zero. 
}
\item 
\end{ex}

\begin{ex}
{\em The Lie groups $S_{3,-1}$ belongs to the family $S_{3,\lambda}$ of  $3$-dimensional solvable complex Lie groups having structure equations 
$$
[Z_1,Z_2]=Z_2\,,\quad  [Z_1,Z_3]=\lambda Z_3\,.
$$
The group is unimodular only for $\lambda=-1$.
}
\end{ex}

\begin{rem}
{\em $4$-dimensional and $5$-dimensional complex solvable Lie groups are classified by Nakamura \cite{Naka}.}
\end{rem}

\section{Uniform Estimates}\label{estimates} 
Let $\hat g$ be  a Chern-flat metric  on a compact complex parellalizable manifold $M$ and let $g=g(t)$ be a solution to the second Chern-Ricci flow  \eqref{Q=0} on $M$ in the interval $[0,\varepsilon]$.  We fix on $M$ a global holomorphic coframe $\{\zeta^1,\dots,\zeta^n\}$ and denote by  $\{Z_1,\dots,Z_n\}$ the dual $(1,0)$-frame. All computations are done with respect to this frames.

\begin{lemma}\label{lemma1}
There exist uniform positive constants $K$ and $\tilde K$ such that 
$$
K\hat g\leq g(t) \leq\tilde  K \hat g
\quad \mbox{for every $t \in [0,\varepsilon]$}\,.
$$
\end{lemma}
\begin{proof}
Working as in \cite {estimates} and \cite{S2} we directly compute  
$$
\begin{aligned}	
(\partial_t-\Delta_g)\,{\rm tr}_{\hat g}g=&(\partial_t-\Delta_g) (\hat g^{\bar s r}g_{s \bar r})=
\hat g^{\bar s r}\dot g_{s \bar r}-\hat g^{\bar s r}g^{\bar lk}g_{r\bar s,k\bar l}
=-\hat g^{\bar s r}\tilde R_{r\bar s}-\hat g^{\bar s r}g^{\bar lk}g_{r\bar s,k\bar l}\\
=&\hat g^{\bar s r}g^{\bar lk}g_{r\bar s,k\bar l}-\hat g^{\bar s r}g^{\bar l k}g^{\bar b a}g_{a\bar l,\bar s }
g_{k	\bar b,r}
-\hat g^{\bar s r}g^{\bar lk}g_{r\bar s,k\bar l}
= -\hat g^{\bar s r}g^{\bar l k}g^{\bar b a}g_{a\bar l,\bar s }g_{k	\bar b,r}\,,
\end{aligned}
$$
hence,
\begin{equation}\label{prepre}
(\partial_t-\Delta_g)\,{\rm tr}_{\hat g}g\leq 0
\end{equation}
and the maximum principle implies 
$$
g\leq \tilde K \hat g
$$
for a uniform constant $K$. Moreover, 
$$
\begin{aligned}
(\partial_t-\Delta_g)\,{\rm tr}_{ g}\hat g=&(\partial_t-\Delta_g) (g^{\bar l k} \hat g_{k \bar l})=
- g^{\bar l \alpha} \dot g_{\alpha\bar \beta} g^{\bar \beta k} \hat g_{k \bar l}+g^{\bar r s}(g^{\bar l \alpha}  g_{\alpha\bar \beta,s} g^{\bar \beta  k})_{,\bar r} \hat g_{k \bar l}\\
=& g^{\bar l \alpha} \tilde R_{\alpha\bar \beta} g^{\bar \beta k} \hat g_{k \bar l}
-g^{\bar r s}g^{\bar l \gamma}g_{\gamma\bar \delta,\bar r}g^{\bar \delta \alpha}  g_{\alpha\bar \beta,s} g^{\bar \beta k} \hat g_{k \bar l}
+g^{\bar l \alpha} \Delta_g  g_{\alpha\bar \beta} g^{\bar \beta k} \hat g_{k \bar l}\\
&-g^{\bar r s}g^{\bar l \alpha}  g_{\alpha\bar \beta,s} g^{\bar \beta \gamma}g_{\gamma \bar \delta,\bar r}g^{\bar 	\delta k} \hat g_{k \bar l}\\
=&g^{\bar l \alpha} g^{\bar qp}g^{\bar b a}g_{a\bar \beta,\bar q }g_{\alpha	\bar b,p} g^{\bar \beta k} \hat g_{k \bar l}
-g^{\bar r s}g^{\bar l \gamma}g_{\gamma\bar \delta,\bar r}g^{\bar \delta \alpha}  g_{\alpha\bar \beta,s} g^{\bar \beta k} \hat g_{k \bar l}\\
&-g^{\bar l \alpha} g^{\bar r s} g^{\bar \beta \gamma} g_{\alpha\bar \beta,s} g_{\gamma \bar \delta,\bar r}g^{\bar 	\delta k} \hat g_{k \bar l}\\
=&-g^{\bar r s}g^{\bar l \gamma}g_{\gamma\bar \delta,\bar r}g^{\bar \delta \alpha}  g_{\alpha\bar \beta,s} g^{\bar \beta k} \hat g_{k \bar l}\leq 0
\end{aligned}
$$
and the maximum principle implies the thesis. 
\end{proof}
Let us denote by $\Gamma$ the tensor 
$$
\Gamma=\Gamma_{ij}^k\zeta^i\otimes \zeta^j\otimes Z_k
$$
and set 
$$
S:=|\Gamma|^2_{g}\,.
$$
Hence we have 
\begin{equation}\label{normgamma}
S=g_{k\bar q}g^{\bar mi}g^{\bar pj}\Gamma_{ij}^k \Gamma_{\bar m\bar p}^{\bar q}=g_{k\bar q}g^{\bar mi}g^{\bar pj}g^{\bar r k}g_{j\bar r,i}g^{\bar q s}g_{s\bar p,\bar m}=
g^{\bar mi}g^{\bar pj}g^{\bar r k}g_{j\bar r,i}g_{k\bar p,\bar m}\,.
\end{equation}

\begin{lemma}
$$
(\partial_t-\Delta_g)S\leq |T|_{g}^2\,S\,.
$$
\end{lemma}
\begin{proof}
We first note that along the flow 
$$
\begin{aligned}
\partial_t(g^{\bar r k}g_{j\bar r,i})=&\,-g^{\bar rs}\dot g_{s\bar m}g^{\bar m k}g_{j\bar r,i}+g^{\bar r k}\dot g_{j\bar r,i}=g^{\bar rs}\tilde R_{s\bar m}g^{\bar m k}g_{j\bar r,i}-g^{\bar r k}\tilde R_{j\bar r,i}\\
=&\,g^{\bar m k}\tilde R_{s\bar m}\Gamma^{s}_{ij}-g^{\bar r k}\tilde R_{j\bar r,i}=-g^{\bar r k}\nabla_i\tilde R_{j\bar r}\,,
\end{aligned}
$$
i.e. 
$$
\partial_t\Gamma_{ij}^k=-\nabla_i\tilde R_{j}^k\,.
$$
It follows that
$$
\begin{aligned}
\partial_t S=&-\tilde R_{k\bar q}g^{\bar mi}g^{\bar pj}\Gamma_{ij}^k \Gamma_{\bar m\bar p}^{\bar q}
+g_{k\bar q}\tilde R_{\bar b}^{\bar m}g^{\bar bi}g^{\bar pj}\Gamma_{ij}^k \Gamma_{\bar m\bar p}^{\bar q}
+g_{k\bar q}g^{\bar mi}\tilde R_{\bar b}^{\bar p}g^{\bar bj}\Gamma_{ij}^k \Gamma_{\bar m\bar p}^{\bar q}\\
&-g_{k\bar q}g^{\bar mi}g^{\bar pj}\nabla_i\tilde R_{j}^k\Gamma_{\bar m\bar p}^{\bar q}-g_{k\bar q}g^{\bar mi}g^{\bar pj}\Gamma_{ij}^k\nabla_{\bar m}\tilde R_{\bar p}^{\bar q}\,.
\end{aligned}
$$
Moreover
$$
\nabla_{\bar b}S=g_{k\bar q}g^{\bar mi}g^{\bar pj}\nabla_{\bar b}\Gamma_{ij}^k\Gamma_{\bar m\bar p}^{\bar q}+g_{k\bar q}g^{\bar mi}g^{\bar pj}\Gamma_{ij}^k\nabla_{\bar b}\Gamma_{\bar m\bar p}^{\bar q}
$$
and 
$$
\begin{aligned}
\Delta_g S=&\, g^{\bar ba}\nabla_a\nabla_{\bar b}S=
g_{k\bar q}g^{\bar mi}g^{\bar pj}g^{\bar ba}(\nabla_a\nabla_{\bar b}\Gamma_{ij}^k)\Gamma_{\bar m\bar p}^{\bar q}+g_{k\bar q}g^{\bar mi}g^{\bar pj}\Gamma_{ij}^k
g^{\bar ba}(\nabla_a\nabla_{\bar b}\Gamma_{\bar m\bar p}^{\bar q})\\
&\,+g_{k\bar q}g^{\bar mi}g^{\bar pj}\,g^{\bar ba}\nabla_{\bar b}\Gamma_{ij}^k\nabla_a\Gamma_{\bar m\bar p}^{\bar q}+g_{k\bar q}g^{\bar mi}g^{\bar pj}g^{\bar ba}\nabla_a\Gamma_{ij}^k\nabla_{\bar b}\Gamma_{\bar m\bar p}^{\bar q}\\
= &|\nabla \Gamma|_g^2+|\bar \nabla \Gamma|_g^2+g_{k\bar q}g^{\bar mi}g^{\bar pj}g^{\bar ba}(\nabla_a\nabla_{\bar b}\Gamma_{ij}^k)
\Gamma_{\bar m\bar p}^{\bar q}+g_{k\bar q}g^{\bar mi}g^{\bar pj}\Gamma_{ij}^k g^{\bar ba}(\nabla_a\nabla_{\bar b}\Gamma_{\bar m\bar p}^{\bar q})\\
=& \,|\nabla \Gamma|_g^2+|\bar \nabla \Gamma|_g^2-g^{\bar mi}g^{\bar pj}g^{\bar ba}\nabla_aR_{i\bar bj	\bar q}\Gamma_{\bar m\bar p}^{\bar q}
+g_{k\bar q}g^{\bar mi}g^{\bar pj}\Gamma_{ij}^kg^{\bar ba}(\nabla_{\bar b}\nabla_a\Gamma_{\bar m\bar p}^{\bar q})\\
&+g_{k\bar q}g^{\bar mi}g^{\bar pj}\Gamma_{ij}^kg^{\bar ba}R_{a\bar b\bar m }^{\bar u}\Gamma_{\bar u\bar p}^{\bar q}
+g_{k\bar q}g^{\bar mi}g^{\bar pj}\Gamma_{ij}^kg^{\bar ba}R_{a\bar b\bar p }^{\bar u}\Gamma^{\bar q}_{\bar m\bar u}-g_{k\bar q}g^{\bar mi}g^{\bar pj}\Gamma_{ij}^kg^{\bar ba}R_{a\bar b\bar u }^{\bar q}\Gamma^{\bar u}_{\bar m\bar p}\\
=& \,|\nabla \Gamma|_g^2+|\bar \nabla \Gamma|_g^2-g^{\bar mi}g^{\bar pj}g^{\bar ba}\nabla_aR_{i\bar bj	\bar q}\Gamma_{\bar m\bar p}^{\bar q}
-g^{\bar mi}g^{\bar pj}\Gamma_{ij}^kg^{\bar ba}\nabla_{\bar b}R_{\bar m a\bar pk}\\
&+g_{k\bar q}g^{\bar mi}g^{\bar pj}\Gamma_{ij}^k\tilde R_{\bar m }^{\bar u}\Gamma_{\bar u\bar p}^{\bar q}
+g_{k\bar q}g^{\bar mi}g^{\bar pj}\Gamma_{ij}^k\tilde R_{\bar p }^{\bar u}\Gamma^{\bar q}_{\bar m\bar u}
-g_{k\bar q}g^{\bar mi}g^{\bar pj}\Gamma_{ij}^k\tilde R_{\bar u }^{\bar q}\Gamma^{\bar u}_{\bar m\bar p}\,.
\end{aligned}
$$
Then using the second Bianchi identity we get
$$
\begin{aligned}
\Delta_g S=& \,|\nabla \Gamma|_g^2+|\bar \nabla \Gamma|_g^2
-g^{\bar mi}g^{\bar pj}g^{\bar ba}\nabla_iR_{a\bar bj	\bar q}\Gamma_{\bar m\bar p}^{\bar q}
+g^{\bar mi}g^{\bar pj}g^{\bar ba}T_{ai}^kR_{k\bar bj\bar q}\Gamma_{\bar m\bar p}^{\bar q}\\
&-g^{\bar mi}g^{\bar pj}\Gamma_{ij}^kg^{\bar ba}\nabla_{\bar m}R_{a \bar b k \bar p}+g^{\bar mi}g^{\bar pj}g^{\bar ba}T_{\bar b\bar m}^{\bar s}R_{\bar s a \bar p k}\Gamma_{ij}^k\\
&+g_{k\bar q}g^{\bar mi}g^{\bar pj}\Gamma_{ij}^k\tilde R_{\bar m }^{\bar u}\Gamma_{\bar u\bar p}^{\bar q}
+g_{k\bar q}g^{\bar mi}g^{\bar pj}\Gamma_{ij}^k\tilde R_{\bar p }^{\bar u}\Gamma^{\bar q}_{\bar m\bar u}-g_{k\bar q}g^{\bar mi}g^{\bar pj}\Gamma_{ij}^k\tilde R_{\bar u }^{\bar q}\Gamma^{\bar u}_{\bar m\bar p}\\
=&  \,|\nabla \Gamma|_g^2+|\bar \nabla \Gamma|_g^2
-g^{\bar mi}g^{\bar pj}g^{\bar ba}\nabla_i\tilde R_{j	\bar q}\Gamma_{\bar m\bar p}^{\bar q}
+g^{\bar mi}g^{\bar pj}g^{\bar ba}T_{ai}^kR_{k\bar bj\bar q}\Gamma_{\bar m\bar p}^{\bar q}\\
&-g^{\bar mi}g^{\bar pj}\Gamma_{ij}^k\nabla_{\bar m}\tilde R_{k \bar p}+g^{\bar mi}g^{\bar pj}g^{\bar ba}T_{\bar b\bar m}^{\bar s}R_{\bar s a \bar p k}\Gamma_{ij}^k\\
&+g_{k\bar q}g^{\bar mi}g^{\bar pj}\Gamma_{ij}^k\tilde R_{\bar m }^{\bar u}\Gamma_{\bar u\bar p}^{\bar q}
+g_{k\bar q}g^{\bar mi}g^{\bar pj}\Gamma_{ij}^k\tilde R_{\bar p }^{\bar u}\Gamma^{\bar q}_{\bar m\bar u}
-g^{\bar mi}g^{\bar pj}\Gamma_{ij}^k\tilde R_{k\bar u }\Gamma^{\bar u}_{\bar m\bar p}\,.
\end{aligned}
$$
Hence 
$$
\begin{aligned}
(\partial_t-\Delta_g) S= &\,\,-|\nabla \Gamma|_g^2-|\bar \nabla \Gamma|_g^2
-g^{\bar mi}g^{\bar pj}g^{\bar ba}T_{ai}^kR_{k\bar bj\bar q}\Gamma_{\bar m\bar p}^{\bar q}-g^{\bar mi}g^{\bar pj}g^{\bar ba}T_{\bar b\bar m}^{\bar s}R_{\bar s a \bar p k}\Gamma_{ij}^k\\
=& \,\,-|\nabla \Gamma|_g^2-|\bar \nabla \Gamma|_g^2
+g^{\bar mi}g^{\bar pj}g^{\bar ba}T_{ai}^kg_{u\bar q}\nabla_{\bar b}\Gamma_{kj}^{u}\Gamma_{\bar m\bar p}^{\bar q}+g^{\bar mi}g^{\bar pj}g^{\bar ba}T_{\bar b\bar m}^{\bar s}g_{k\bar u}\nabla_a\Gamma_{\bar s\bar p}^{\bar u}\Gamma_{ij}^k\\
\leq & -|\nabla \Gamma|_g^2-|\bar \nabla \Gamma|_g^2+2 |T|_{g}S^{1/2}|\bar \nabla \Gamma|_g\\
\leq & -|\nabla \Gamma|_g^2+ |T|_{g}^2S\,,
\end{aligned}
$$
and finally
$$
(\partial_t-\Delta_g) S\leq |T|_{g}^2\,S\,,
$$
as required. 
\end{proof}

\begin{cor}
We have 
$$
(\partial_t-\Delta_g)S\leq 2| \hat{T}|^2_{g}S+8S^2\,,
$$
where $\hat T$ is the torsion of $\hat g$.
\end{cor}
\begin{proof}
Recalling that $\hat g$ is Chern flat we have that the components of the torsion of $\hat \nabla$ coincide with the components of the brackets with respect to the chosen frame, thus  
$$
T_{ij}^k=\Gamma_{ij}^k-\Gamma_{ji}^k-\hat T_{ij}^k\,. 
$$
Hence
$$
|T|_g\leq 2S^{1/2}+|\hat T|_g
$$
which implies 
$$
|T|^2_g\leq | \hat{T}|^2_g+4S+4  | \hat{T}|_g S^{1/2}\leq 2| \hat{T}|^2_g+8S\,,
$$
and the claim follows. 
\end{proof}

\section{Proof of Theorem \ref{main1}}\label{3}
Let $g(t)$, $t\in [0,\varepsilon]$, be a solution to \eqref{Q=0} and fix a background Chern-flat metric $\hat g$ on $M$. 
In this section we use the same setting adopted in the section \ref{estimates}. In view of Lemma \ref{lemma1} there exist positive uniform constants $K,	\tilde K$ such that 
$$
K\hat g\leq g(t)\leq \tilde K\hat g\,,
$$ 
for every $t\in [0,\varepsilon]$.
Now using the computation of Lemma \ref{lemma1} we have
$$
(\partial_t-\Delta_g)\,{\rm  tr}_{\hat g} g \leq -\hat g^{\bar s r}g^{\bar l k}g^{\bar b a}g_{a\bar l,\bar s }g_{k	\bar b,r}\leq -K g^{\bar s r}g^{\bar l k}g^{\bar b a}g_{a\bar l,\bar s }g_{k	\bar b,r} = - KS \,. 
$$

Let  
 $$
G=\log S+ b\,{\rm  tr}_{\hat g} g\,,
$$
where $b$ is a positive constant to be determined later. 

We have 
$$
\begin{aligned}
 (\partial_t-\Delta_g)\, G & = \frac{ (\partial_t-\Delta_g)S}{S} + \frac{|\nabla S|_g^2}{S^2} + b(\partial_t-\Delta_g)\, {\rm  tr}_{\hat g} g
 \\ &\leq \,2| \hat{T}|^2_g+\frac{|\nabla S|_g^2}{S^2}+(8-bK) S\,.
\end{aligned}
$$
Assume that $G$ has a maximum point  $(\hat x,\hat t)\in M\times [0,\varepsilon]$ with $\hat t>0$. Then 
$$
0\leq 2| \hat{T}|^2_g+\frac{|\nabla S|_g^2}{S^2}+(8-bK) S\quad \mbox{at $(\hat x,\hat t)$}.
$$

From   $\nabla G(\hat x,\hat t)=0$ we have 
%$$
%\frac{\nabla S}{S} = -b \nabla {\rm tr}_{\hat g} g \quad \mbox{at $(\hat x,\hat t)$},
%$$
%and 
$$
\frac{|\nabla S|_{g}^2}{S^2} = b^2 |\nabla {\rm tr}_{\hat g} g|^2_g \quad \mbox{at $(\hat x,\hat t)$}.
$$
Now, since 
$$
\begin{aligned}
\nabla_{k}\,{\rm tr}_{\hat g}g=&\,\nabla_k (\hat g^{\bar rs}g_{s\bar r})\\
=&\, (\nabla_k \hat g^{\bar rs})g_{s\bar r}\\
=&\,\hat g^{\bar ba}\Gamma_{ka}^l g_{l\bar b}\,.
\end{aligned}
$$
We have 
$$
\begin{aligned}
|\nabla\,{\rm tr}_{\hat g}g|_g^2=&\, g^{\bar pk} \hat g^{\bar ba}\Gamma_{ka}^l g_{l\bar b}\hat g^{\bar sr}
\Gamma_{\bar p\bar s}^{\bar q}g_{r\bar q}\\
\leq &\, K^2 g^{\bar pk}g^{\bar ba}\Gamma_{ka}^l g_{l\bar b}g^{\bar sr}
\Gamma_{\bar p\bar s}^{\bar q}g_{r\bar q}
\leq \, K^2 g^{\bar p k} \Gamma_{ka}^a 
\Gamma_{\bar p\bar q}^{\bar q}\\
= &\,K^2|w|_{g}^2\,.
\end{aligned}
$$
Hence if $|w|_g$ is uniformly bounded we have that 
$|\nabla\,{\rm tr}_{\hat g}g|_g^2$ and $\tfrac{|\nabla S|_{g}^2}{S^2}$ are bounded ndependently of $b$. Thus 
$$
(bK-8) S \leq (1+b^2)C\quad \mbox{at $(\hat x,\hat t)$}
$$
for a uniform constant $C$, and for $b>\tfrac{8}{K}$ we have 
$$
S(\hat x,\hat t)\leq C\,,
$$ 
where $C$ does not depend on $g$. 

Hence 
$$
G(x,t)\leq G(\hat x,\hat t)=\log S(\hat x,\hat t)+ b{\rm  tr}_{\hat g} g(\hat x,\hat t)
\leq \log C+bn\tilde K
$$
and 
$$
\log S \leq \log C+bn\tilde K \mbox{ in } M\times [0,\varepsilon]\,.
$$
This proves that $g(t)$ has a $C^1$ a priori bound in $M\times [0,\varepsilon]$.  

\medskip 
Now we can work as in the final step of the proof of \cite[Theorem 6]{estimates}. Our equation can be regarded as a quasi-linear parabolic system and a classical theorem by Ladyzenskaja et al. (\cite{Lady} Theorem 5.1 in Chapter 7, page 586)
implies that 
$$
\|g\|_{C^{2+\alpha,1+\alpha/2}(M\times [0,\varepsilon])}\leq C\,,
$$ 
where $C$ does not depend on $g$. Hence from the parabolic Schauder estimates we deduce higher order estimates on $g$. Then Ascoli-Arzel\`a theorem can be used to prove that actually $g(t)$ can  be extended in $M\times [0,\infty)$. 

\section{Stability: proof of Theorem \ref{stability}}

Now we prove Theorem \ref{stability} showing that Chern-flat metrics on compact complex parallelizable manifolds are 
dynamically stable. 

\medskip 
Fix a background left-invariant metric $\hat g$ on $M$. Without loss of generality we may assume ${\rm vol}_{\hat g}(M)=1$.

\begin{lemma}\label{L2} There exists $\delta>0$ such that if $g(t)$ is a solution of \eqref{Q=0} for $t\in [0,\varepsilon]$ and 
$$
\|g(t)-\hat g\|_{C^{2}_{\hat g}}\leq \delta \quad \mbox{ for every $t\in 	[0,\varepsilon]$}\,,
$$
then 
$$
\|\tilde R(g(t))\|_{L^2_{\hat g}} \leq a {\rm e}^{-b\lambda t} \quad \mbox{ for every $t \in [0,\varepsilon]$\,. }
$$ 
Here $\lambda$ is the first positive eigenvalue of $-\Delta_{\hat g}$, $a$ and $b$ are positive constants depending only on $\delta$ and $a\to 0$, $b\to 1$ as $\delta \to 0^+$. 
\end{lemma}

\begin{proof}

We have 
$$
\partial_tR_{i\bar jk\bar l}=R_{i\bar jk}^m\dot g_{m\bar l}-\nabla_{\bar j}\nabla_{i}\dot g_{k\bar l}\,,
$$
and so 
$$
\partial_{t}\tilde R_{k\bar l}=\partial_{t}(g^{\bar ji} R_{i\bar j k\bar l})=
-g^{\bar ja}\dot g_{a	\bar b}g^{\bar b i} R_{i\bar j k\bar l}+g^{\bar ji} R_{i\bar jk}^m\dot g_{m\bar l}-g^{\bar ji}\nabla_{\bar j}\nabla_{i}\dot g_{k\bar l}\,.
$$
It follows  
$$
\partial_{t}\|\tilde R\|^2_{L^2_{\hat g}}
=2 \int_M (g^{\bar ja}\tilde R_{a	\bar b}g^{\bar b i} R_{i\bar j k\bar l}-g^{\bar ji} R_{i\bar jk}^m\tilde R_{m\bar l}+
g^{\bar ji}\nabla_{\bar j}\nabla_{i}\tilde R_{k\bar l})\hat g^{\bar lr}\hat g^{\bar s k}\tilde R_{r\bar s}\,dV_{\hat g}
$$
and 
for 
$$
\|g(t)-\hat g\|_{C^2_{\hat g}}\leq \delta\,,\quad \mbox{ for every $t\in 	[0,\varepsilon]$}\,,
$$
we have 
$$
\partial_{t}\|\tilde R\|^2_{L^2_{\hat g}}
\leq A \|\tilde R\|_{L^2_{\hat g}}^2
+B\hat g^{\bar lr}\hat g^{\bar s k}\int_M ( \Delta_{\hat g}\tilde R_{k\bar l})\,\tilde R_{r\bar s}\,dV_{\hat g}\,=A\|\tilde R\|_{L^2_{\hat g}}^2 +B\langle \Delta_{\hat g}\tilde R,\tilde R\rangle_{L^2_{\hat g}}\,,
$$
where $A,B$ depend on $\delta$ and 
$A \to 0$, $B\to 1$ if $\delta \to 0^+$. 
Let $\tilde R^0$ be the left-invariant tensor whose components with respect the fixed frame  are 
$$
\tilde R^0_{k\bar l}=\int_M\tilde R_{k\bar l}\,dV_{\hat g}\,.
$$
We observe that $\tilde R^0$ is $\Delta_{\hat g}$-harmonic and  $\tilde R-\tilde R^0$ is 
$L^2_{\hat g}$-orthogonal to the kernel of $\Delta_{\hat g}$. Then 
$$
\begin{aligned}\langle \Delta_{\hat g}\tilde R,\tilde R\rangle_{L^2_{\hat g}}=
\langle \Delta_{\hat g}(\tilde R-\tilde R^0),\tilde R-\tilde R^0\rangle_{L^2_{\hat g}}\leq -\lambda \|\tilde R-\tilde R^0\|_{L^2_{\hat g}}^2\leq  -\lambda (\|\tilde R\|_{L^2_{\hat g}}-\|\tilde R^0\|_{L^2_{\hat g}})^2\,,
\end{aligned}
$$
where $\lambda$ is the first positive eigenvalue of $-\Delta_{\hat g}$. 
Moreover 
$$
\begin{aligned}
\|\tilde R^0\|_{L^2_{\hat g}}^2=&\int_M \hat g^{\bar s r}\hat g^{\bar l k}\tilde R^0_{r\bar l}\tilde R^0_{k\bar s}\,dV_{\hat g}=\hat g^{\bar s r}\hat g^{\bar l k}\left(\int _M  \tilde R_{r\bar l}\,dV_{\hat g}\right)
\left(\int _M  \tilde R_{k\bar s}\,dV_{\hat g}\right)\\
=&\hat g^{\bar s r}\hat g^{\bar l k}\left(\int _M  (\tilde R_{r\bar l}-\Delta_{\hat g}g_{r\bar l})\,dV_{\hat g}\right)
\left(\int _M  (\tilde R_{k\bar s}-\Delta_{\hat g}g_{k\bar s})\,dV_{\hat g}\right)\\
\leq & \|	\tilde R-	\Delta_{\hat g}g\|_{L^2_{\hat g}}^2\,.
\end{aligned}
$$
Now since $g\mapsto \tilde  R(g)$ is Fr\'echet differentiable, for every $\kappa>0$, if we choose $\delta$ small enough we have 
$$
\|\tilde R-\Delta_{\hat g}g \|_{L^2_{\hat g}	
}\leq \kappa \|g-\hat g\|_{W^{2,2}_{\hat g}}\,.
$$
In particular we can fix $\kappa=1$ and 
\begin{equation}\label{k=1}
\|\tilde R-\Delta_{\hat g}g \|_{L^2_{\hat g}	
}\leq \|g-\hat g\|_{W^{2,2}_{\hat g}}
\end{equation}
for $\delta$ small enough. 

Moreover elliptic regularity yields that  
$$
\|g-\hat g\|_{W^{2,2}_{\hat g}}\leq C\|\Delta_{\hat g}g\|_{L^2_{\hat g}}\,.
$$
Now 

$$
\|\tilde R\|_{L^2_{\hat g}}\geq \|\Delta_{\hat g}g\|_{L^2_{\hat g}}- \|\tilde R-\Delta_{\hat g}g\|_{L^2_{\hat g}}
\geq (1- C)  \|\Delta_{\hat g}g\|_{L^2_{\hat g}}\,,
$$
i.e.
$$
  \|\Delta_{\hat g}g\|_{L^2_{\hat g}}\leq \frac{1}{1- C}\|\tilde R\|_{L^2_{\hat g}}\,.
$$
Therefore 
 $$
\|\tilde R-\Delta_{\hat g}g \|_{L^2_{\hat g}}^2\leq \frac{C^2}{(1-C)^2}\|\tilde R\|_{L^2_{\hat g}}^2
$$ 
and so 
$$
\partial_{t}\|\tilde R\|^2_{L^2_{\hat g}}
\leq A \|\tilde R\|_{L^2_{\hat g}}^2 -\lambda B \|\tilde R\|_{L^2_{\hat g}}^2+\lambda B \|\tilde R^0\|_{L^2_{\hat g}}^2\leq  
\left(A -\lambda B+\frac{\lambda B  C^2}{(1- C)^2}\right)\|\tilde R\|_{L^2_{\hat g}}^2\,.
$$
Since \eqref{k=1} holds true if we shrink $\delta$, and $A\to 0$ and $B
\to 1$ as $\delta \to 0$, up to shrinking $\delta$ we may assume 
$$
A -\lambda B+\frac{\lambda B C^2}{(1- C)^2}<0\,.
$$  
Then choosing 
$$
a=\|\tilde R\|_{L^2_{\hat g}}^2 \,,\quad \mbox{ and } b= -\frac{A}{\lambda} +B-\frac{ B C^2}{(1- C)^2}\,,
$$
we get the claim applying Gronwall's lemma.
\end{proof}

\begin{lemma} 
	\label{inter}
	There exists $\delta>0$ such that if $g(t)$ is a solution of \eqref{Q=0} for $t\in [0,\varepsilon]$ and 
$$
\|g(t)-\hat g\|_{C^2_{\hat g}}\leq\delta \quad \mbox{ for every $t\in [0,\varepsilon]$}\,,
$$
then for every $m \in \mathbb N$ and $\tau \in (0,\varepsilon)$
\begin{equation}
\label{Hndecay}
\|\tilde R(g(t))\|_{H^{m}_{\hat g}} \leq C  \|\tilde R(g(0))\|_{L^2_{\hat g}}{\rm e}^{{-C} t}\quad \mbox{ for every } t \in [\tau,\varepsilon]\,.
\end{equation}x
The constant $C$ depends only on $\varepsilon,m,\tau $, $\delta$ and $\hat g$. 
\end{lemma}

\begin{proof}
Using Hamilton's interpolation theorem for tensors \cite[Corollary 12.7]{Hamilton} we have that for every tensor $Q$ and $m>0$ 
$$
\int_M |\nabla^m Q|_g^2\,dV_{\hat g}\leq C\left(\int_M |\nabla^{2m}Q|^2_{\hat g}\,dV_{\hat g}\right)^{\frac{1}{2}}
\|Q\|_{L^2_{\hat g}}\,,
$$
for a constant $C>0$ depending only $m$ and $\dim M$. Thus in our case 
$$
\int_M |\nabla^m \tilde R(g(t))|_{\hat g}^2\,dV_{\hat g}\leq C\left(\int_M |\nabla^{2m}\tilde R(g(t))|^2_{\hat g}\,dV_{\hat g}\right)^{\frac{1}{2}}
\|\tilde R(g(t))\|_{L^2_{\hat g}}\,.
$$
Assume $\|g(t)-\hat g\|_{C^2}\leq \delta' $ for every $t\in [0,\varepsilon]$, then 
$$
\|R(g(t))\|_{C^0_{g(t)}}\leq K,\quad \|T(g(t))\|_{g(t)}^2\leq K,\quad 
\|\nabla T(g(t))\|_{g(t)}^2\leq K
$$
for a constant $K$ which depends on un upper bound on $\delta'$. Using Shi-type estimates for Hermitian curvature flows 
\cite[Theorem 7.3]{HCF} we have 
$$
\|\nabla^mR(g(t))\|_{C^0_{\hat g}}\leq \frac{CK}{t^{m/2}}
$$
for every $t\in(0,\varepsilon]$, where $C$ depends on $m, K$ and 
$\varepsilon$.  In particular for $\tau \in (0,\varepsilon)$ we have
 $$
\|\nabla^mR(g(t))\|_{C^0_{g(t)}}\leq C \quad \mbox{ for every $t \in [\tau,\varepsilon]$}
$$
where $C$ depends on $m$, $\tau$, $K$ and $\varepsilon$.  It follows that 
$$
\|R(g(t))\|_{H^m_{g(t)}}\leq C \|\tilde R(g(t))\|_{L^2_{g(t)}}^{1/2} \mbox{ for every $t\in [\tau,\varepsilon]$}
$$
where $C$ depends  only on $\varepsilon,m$ and an upper bound on $\delta'$. By  Lemma \ref{L2} we find $0<\delta<\delta'$ such that if $\|g(t)-\hat g\|_{C^{2}_{\hat g}}\leq \delta$ we have
$$
\|\tilde R(g(t))\|_{H^m_{\hat g}}\leq C a {\rm e}^{-b\lambda t/2} \mbox{ for every $t\in [\tau,\varepsilon]$}\,.
$$
where $a\to 0$, $b\to 1$ as $\delta \to 0^+$. The claim follows taking $\delta$ small enough. 
\end{proof}

\begin{proof}[Proof of Theorem $\ref{stability}$]

Let $\delta>0$ be as in lemma \ref{inter}.
Take $\delta_0>0$ such that if  
$$
\|g_0-\hat g\|_{H^{n+3}_{\hat g}}< \delta_0
$$
we have 
$$
\|\tilde R(g_0)\|_{L^{2}_{\hat g}}+\delta_0 <\frac{\delta}{C_n}\,, 
$$
where $C_n$ is the constant given by the Sobolev embedding theorem such that $\|f\|_{C^{2,\alpha}}\leq C_n \|f\|_{H^{n+3}}$ for every symmetric 2-tensor $f$. Let $g$ be the maximal time solution to \eqref{Q=0} with initial datum $g_0$ and let  
$$
\varepsilon=\sup\{s\in [0,\infty)\,\,|\,\, \mbox{for $t\in [0,s]$ } g(t)\mbox{ exists and }\|g(t)-\hat g\|_{C^2_{\hat{g}}}\leq \delta \}.
$$
By the short time existence of the flow $\varepsilon>0$. 
Assume  by contradiction $\varepsilon<\infty$. Note that in  this case $g(t)$ exists for $t \in [0,\varepsilon]$ and

$$
\|g(\varepsilon)-\hat g\|_{C^2_{\hat{g}}}=\delta\,.
$$
Let $\tau \in (0,\varepsilon)$ be such that 
$$
\|g(\tau)-\hat g\|_{H^{n+3}_{\hat{g}}}< \delta_0\,.
$$
Then we have 
$$
\begin{aligned}
\|g(\varepsilon)-\hat g\|_{C^2_{\hat{g}}}\leq &\,C_{n}\|g(\varepsilon)-\hat g\|_{H^{n+3}_{\hat{g}}}\\
\leq &\, C_{n}(\|g(\varepsilon)-g(\tau)\|_{H^{n+3}_{\hat{g}}}+\|g(\tau)-\hat g\|_{H^{n+3}_{\hat{g}}})\\
\leq &\, C_n\left(\int_{\tau}^\varepsilon \|\tilde R(g(s)\|_{H^{n+3}_{\hat g}}\,ds+\delta_0\right)\\
\leq &\, C_n \left(\int_{\tau}^\varepsilon C{\rm e}^{-Ct}\|\tilde R(g_0)\|_{L^{2}_{\hat g}}\,ds+\delta_0\right)\\
\leq &\, C_n \left({\rm e}^{-C\tau}\|\tilde R(g_0)\|_{L^{2}_{\hat g}}+\delta_0\right)\\
\leq &\, C_n \left(\|\tilde R(g_0)\|_{L^{2}_{\hat g}}+\delta_0\right)\\
<&\,\delta
\end{aligned}
$$
which gives us a contradiction. It follows that $g(t)$ exists for every $t \in [0,\infty)$ and
$$
\|g(t)-\hat g\|_{C^2_{\hat g}}\leq \delta \quad \mbox{ for every $t\in [0,\infty)$}\,,
$$
and Lemma \ref{inter} implies that for every $m\in\mathbb N$, $\|\tilde R(g(t))\|_{H^{m}_{	\hat g}}	\to 0$ as $t\to \infty.$ 

Now we show that for every $m\in \mathbb N$,  $g(t)$ converges to a Chern-flat metric in $H^m_{\hat g}$ norm. Let 
$$
g_{\infty}=\int_{0}^{\infty}\tilde R(g(s))\,ds. 
$$
For every $t>0$, using again Lemma \ref{inter} we have 
$$
\|g(t)-g_{\infty}\|_{H^{m}_{\hat g}}=\left\|\int_{t}^{\infty}\tilde R(g(s))\,ds\right\|_{H^{m}_{\hat g}}\leq 
\int_{t}^{\infty}\left\|\tilde R(g(s))\right\|_{H^{m}_{\hat g}}\,ds
\leq {\rm e}^{-Ct}\left\|\tilde R(g(0))\right\|_{L^{2}_{\hat g}}
$$
which implies that $g(t)\to g_{\infty}$ in $C^\infty$-topology. Moreover, since 
$$
\tilde R(g_\infty)=\lim_{t\to 	\infty }\tilde R(g(t))=0\,.
$$

Finally we remark that a Hermitian metric $g$ on a compact complex parallizable manifold having $\tilde R=0$ is always Chern-flat. Indeed from \eqref{RicDelta}  we 
have that equation $\tilde R=0$ implies 
$$
\Delta_g {\rm tr}_{\hat g}g-\hat g^{\bar r s}g^{\bar lk}g^{\bar b a}g_{a\bar s,\bar l }g_{r	\bar b,k}=0\,.
$$
Hence the maximum principle yields that 
$$
\hat g^{\bar r s}g^{\bar lk}g^{\bar b a}g_{a\bar s,\bar l }g_{r	\bar b,k}=0
$$
and that the components of $g$ are consequently constant. 
\end{proof}

\section{Remarks}
\label{conv}

In this last we give some remarks on the behaviour of the flow on some class of examples.

\begin{prop}\label{scase}
Let $M$ be a compact complex parallelizable manifold and let $g_0$ be a Hermitian metric which is diagonal with respect to a left-invariant $(1,0)$-frame. Then the maximal time solution $g$ to \eqref{Q=0} starting from $g_0$ is still diagonal for every $t$, is defined for $t\in [0,\infty)$ and converges to a Chern-flat metric in $C^{\infty}$-topology.  
\end{prop}

\begin{proof}
 From formula \eqref{RicDelta} it follows that if $g$ is a diagonal Hermitian metric, then also $\tilde R$ is diagonal and its components $\tilde R_{r \bar r}$ satisfy
$$
\tilde R_{r	\bar r}=-\Delta_g g_{r\bar r}+\sum_k\frac{1}{g_{k\bar k}\,g_{r\bar r}}\,|g_{r\bar r,k}|^2=-\sum_k
\frac{1}{g_{k\bar k}} g_{r\bar r,k\bar k}+\sum_k\frac{1}{g_{k\bar k}\,g_{r\bar r}}\,|g_{r\bar r,k}|^2=-g_{r\bar r}\Delta_g \log g_{r\bar r}  \,.
$$

Let $g$ be the maximal time solution to \eqref{Q=0} on $M$
with starting point a diagonal metric $g_0$. Hence $g(t)$ is diagonal for every $t$ and its components evolve as 
$$
 \dot g_{r	\bar r}=g_{r\bar r}\Delta_g \log g_{r	\bar r}\,,
$$ 
i.e.
$$
\partial_t\log g_{r\bar r}=\Delta_g \log g_{r\bar r}\,. 
$$

Let $f_r=\log g_{r\bar r}$, then 
\begin{equation}\label{log}
\dot f_r=\Delta_g f_r\,.
\end{equation}
Now we consider the Chern-flat metric $\hat g$ which is the identity in the fixed left-invariant frame in which $g$ is diagonal.
Let us consider the quantities
$$
a_r:=\int_M |\hat \nabla f_r|_{\hat g}^2\,dV_{\hat g}
$$
and compute their evolution along the flow taking into account that $\hat g$ is balanced:
$$
\begin{aligned}
\dot a_r & =2\int_M \langle \hat\nabla  \dot f_r,\hat \nabla f_r\rangle_{\hat g}\,dV_{\hat g}
=2\int_M \langle \hat \nabla(  \Delta_{g} f_r),\hat \nabla f_r \rangle_{\hat g}\,dV_{\hat g}=-2 \int_M \Delta_{g} f_r \Delta_{\hat g} f_r \,dV_{\hat g}\\
& \leq -2 \tilde K^{-1}\int_M (\Delta_{\hat g} f_r)^2 \,dV_{\hat g}\,,
\end{aligned}
$$
where $\tilde K>0$ is the lower bound of Lemma \ref{lemma1}.
In particular for every $r$ the function $a_r$ is non-increasing and $g$ is defined in $M\times [0,\infty)$. 
Now we want to compute $\partial_t\dot a_r$. In order to do that we first compute

$$
\begin{aligned}
\partial_t \dot f_r & =\partial_t \Delta_g f_r = 
\partial_t (g^{k \bar k} f_{r, k \bar k}) = - g^{k \bar k} \dot g_{k \bar k} g^{k \bar k} f_{r, k \bar k} + g^{k \bar k}\dot f_{r, k \bar k} \\
& =  - g^{k \bar k} g_{k \bar k} g^{k \bar k} f_{r, k \bar k} \dot f_k + g^{k \bar k}\dot f_{r, k \bar k} = - \dot f_k g^{k \bar k} f_{r, k \bar k}  + \Delta_g \dot f_r\,.
\end{aligned}
$$

Therefore 

$$
\begin{aligned}
	\partial_t \dot a_r & =-2 \partial_t \int_M \Delta_{g} f_r \Delta_{\hat g} f_r \,dV_{\hat g} \\
	&= 2 \int_M  \dot f_k g^{k \bar k} f_{r, k \bar k} \Delta_{\hat g} f_r \,dV_{\hat g} - 2  \int_M \Delta_{g} \dot f_r \Delta_{\hat g} f_r \,dV_{\hat g} - 2 \int_M \dot f_r \Delta_{\hat g} \dot f_r \,dV_{\hat g} \\
%& = 2 \int_M  \dot f_k g^{k \bar k} f_{r, k \bar k} \Delta_{\hat g} f_r \,dV_{\hat g} - 2  \int_M \Delta_{g} \dot f_r \Delta_{\hat g} f_r \,dV_{\hat g} - 2 \int_M \dot f_r \Delta_{\hat g} \dot f_r \,dV_{\hat g} \\
& \geq 2 \int_M  \dot f_k g^{k \bar k} f_{r, k \bar k} \Delta_{\hat g} f_r \,dV_{\hat g} - 2 K \int_M \Delta_{g} \dot f_r \dot f_r \,dV_{\hat g} \geq 2 \int_M  \dot f_k g^{k \bar k} f_{r, k \bar k} \Delta_{\hat g} f_r \,dV_{\hat g}\,.
\end{aligned}
$$
By the results of the previous section we have $\dot f_{k} \geq -C$ for a uniform constant. Hence 
$$
\partial_t \dot a_r\geq -2C\int_M  g^{k \bar k} f_{r, k \bar k} \Delta_{\hat g} f_r \,dV_{\hat g}=-2C\int_M  \Delta_g f_{r}\, \Delta_{\hat g} f_r \,dV_{\hat g} = C\dot a_r\,. 
$$	

Hence every $a_r$ is a decreasing non-negative smooth function such that 
$$
\ddot  a_r \geq C\dot a_r
$$
for a uniform positive constant $C$. By Gronwall's lemma it follows that 
$\dot a_r(t)\geq \dot a_r(s){\rm e}^{C(t-s)} $ for every $0\leq s<t$. Setting $b(t)=-\dot a(t)$ we have that $b(t) \leq b(s) {\rm e}^{C(t-s)}$ for every $0 \leq s<t$.
Now the fact that $a$ is nonnegative implies that $\int_0^\tau b(t) dt \leq a(0)$ for every $\tau>0$. Hence $\sum_{m=0}^\infty \int_m^{m+1} b(t)dt$ converges. Thus there is a sequence of times $t_m \in [m,m+1)$ such that $b(t_m) \to 0$. So we have $b(t) \leq b(t_{m-1})e^{2C}$ for all $t \in [m,m+1)$ and hence $\dot a(t)=- b(t) \to 0$ as $t \to \infty$
(Note that a similar argument is used in \cite[Section 2]{PS}).
%and $\lim_{t\to\infty}\dot a_r(t)=0$. 

We now show that each $f_r$ converges to the constant function $\bar f_r\equiv  \int_M f_r(0)\,dV_{\hat g}$ as $t\to \infty$. Any sequence $t_k\to \infty$ in $[0,\infty)$ admits a subsequence $t_{j_k}$ such that 
$f_r(t_{j_{k}})$ converges in $C^{\infty}$ topology to a smooth function $\bar f_{r}$. Since $\lim_{t\to\infty}\dot a_r(t)=0$, $\bar f_{r}$ is constant. Integrating equation \eqref{log}, we deduce that $\int_M f_r\,dV_{\hat g}$ is constant along the flow. Hence $\bar f_{r}\equiv \int_M f_r(0)\,dV_{\hat g}$. It follows that $f_r$ converges in $C^{\infty}$-topology to $\bar f_{r}$, as required.

Hence $g(t)$ converges as $t\to \infty$ in $C^{\infty}$-topology to a metric with constant entries in the fixed left-invariant $(1,0)$-frame and the claim follows. 
\end{proof}

\begin{rem}{\em A special case of Proposition \ref{scase} is when the initial metric is conformally Chern-flat, i.e. when $g_0=f_0\,\hat g$ with $\hat g$ Chern-flat and $f_0\in C^{\infty}(M,\mathbb{R})$ is positive. In this case the solution to \eqref{Q=0} evolves as $g=f\,\hat g$ with $f\in C^{\infty}(M\times [0,\infty),\mathbb R_+)$ and $f$ converges in $C^{\infty}$-topology to the constant function $f_{\infty}\equiv \int_M f(0)\,\,dV_{\hat g}$. }
\end{rem}
 \begin{rem}{\em An easy case in which one can prove long-time existence and convergence of the flow \eqref{Q=0} is the following.
Let $M$ be the total space of a complex  bundle over a torus. Fix a $(1,0)$-frame on $M$ such that the first vectors are tangent to the basis.  Let $g$ be a Hermitian metric on $M$ taking the following expression 
\begin{equation}\label{matrix}
g=\begin{pmatrix}
h & 0\\
0 &  k
\end{pmatrix}
\end{equation}
with $h$ metric on the base and $k$ constant. Then 
$$
\tilde R_g=
\begin{pmatrix}
\tilde{R}_h & 0\\
0 &  0
\end{pmatrix}
$$
where $\tilde{R}_h$ is the the second Chern-Ricci form of $h$. Therefore if the initial metric takes expression \eqref{matrix}, then flow \eqref{Q=0} evolves the component of $g$ along the basis via flow \eqref{Q=0} restricted to the basis , while it fixes the component along the fibers.  In particular if $h$ is K\"ahler the flow on the base is the K\"ahler-Ricci flow and $g$ converges to a Chern-flat metric.  
}
\end{rem}

\end{document}